\documentclass[12pt]{article}

\usepackage[english]{babel}
\IfFileExists{tgtermes.sty}{\usepackage{tgtermes}}{}
\IfFileExists{MnSymbol.sty}{\usepackage{MnSymbol}}{}

\usepackage[title, titletoc, page]{appendix}

\usepackage{amsmath}
\usepackage{amsthm}

\usepackage{amssymb}
\usepackage{bm}

\usepackage{url}
\usepackage{hyperref}
\usepackage{xcolor}

\newtheorem*{theorem*}{Theorem}
\newtheorem{theorem}{Theorem}[section]
\newtheorem{proposition}{Proposition}[section]
\newtheorem{lemma}{Lemma}[section]

\newtheorem{innercustomthm}{Theorem}
\newenvironment{customthm}[1]
  {\renewcommand\theinnercustomthm{#1}\innercustomthm}
  {\endinnercustomthm}

\newcommand{\Mod}[1]{\ (\mathrm{mod}\ #1)}
\newcommand{\Rea}{\operatorname{Re}\ }

\def\keywords{\vspace{.5em}
{\textbf{Keywords}:\,\relax%
}}

\def\MSC{\vspace{.5em}
{\textbf{Mathematics Subject Classification}:\,\relax%
}}

\newcommand{\be}{\begin{equation}}
\newcommand{\ee}{\end{equation}}
\newcommand{\bea}{\begin{eqnarray}}
\newcommand{\eea}{\end{eqnarray}}

\begin{document}
	\begin{center}
		\bf \Large Lower bounds on the measure of the support of positive and negative parts of trigonometric polynomials
	\end{center}

	\begin{center}
		Abdulamin Ismailov\footnote{National Research University Higher School of Economics, Moscow, Russia, \href{mailto:}{\nolinkurl{arismailov@edu.hse.ru}}}
	\end{center}

	\begin{abstract}
		The measure of the positivity set $\{ x \in [0; 2\pi] \mid f(x) > 0 \}$ of a trigonometric polynomial $f$ is bounded from below by the Motzkin density. We generalize the bound to polynomials in several variables and almost periodic functions. We then use these generalizations to extend known results on Taikov's problem.

		\keywords{trigonometric polynomials, positivity set, Taikov's problem, Motzkin density}

		\MSC{Primary 42A05 $\cdot$ 42A75 $\cdot$ 26D05 $\cdot$ 05E99 $\cdot$ 05C69}
	\end{abstract}

	\tableofcontents
	\newpage

	\section{Introduction}

	In the present paper we study lower bounds on the measure of the positivity and negativity sets of trigonometric polynomials. To this end we define a quantity
	\[ \gamma_+(f) = \frac{1}{2 \pi} \mu(\{ x \in [0; 2\pi] \mid f(x) > 0 \} ), \]
	where $\mu$ is the Lebesgue measure on $\mathbb{R}$ and $f$ is a non-zero periodic trigonometric polynomial
	\[ f(x) = \sum_{k = 1}^{n} [ a_k \cos(kx) + b_k \sin(kx) ]. \]
	Note that because of the zero constant term, the integral
	\[ \int_{0}^{2\pi} f(x) dx = 0 \]
	vanishes and, consequently, the quantity $\gamma_+(f)$ is greater than zero.

	L. V. Taikov \cite{Taikov} proposed the following problem: for a degree $n$ non-zero trigonometric polynomial $f$ with zero constant term what is the least measure of the positivity set? In the area of combinatorics problems like Taikov's have applications (see \cite{MatolcsiRuzsa}, \cite{XuYip}). To give an example a degree $n$ trigonometric polynomial
	\[ t_n(x) = \frac{\cos\left( \frac{n+1}{2} x \right)^2}{\cos(x) - \cos\left( \frac{\pi}{n+1} \right)} \]
	has $\gamma_+(t_n) = \frac{1}{n+1}$. As mentioned in \cite{Taikov} S. B. Stechkin and N. I. Chernykh conjectured that the minimum of $\gamma_{+}(f)$ is attained by $t_n$. A. G. Babenko \cite{Babenko} has proven this conjecture (see also \cite{Tabachnikov}). Later V. A. Yudin \cite{Yudin} extended this result to trigonometric polynomials of the form
	\begin{equation} \label{form_m_n}	f(x) = \sum_{k = m}^{n} [a_k\cos(kx) + b_k\sin(kx)], \quad m > 0, \end{equation}
	by showing that the minimum of $\gamma_+(f)$ is $\frac{m}{n + m}$. Moreover, A. Ulanovskii \cite{Ulanovskii} has shown that a non-zero function $f$ admitting a representation
\[ f(x) = \operatorname{Re}\left( \int_c^d e^{isx}\,d\sigma(s) \right). \]
	where $\sigma$ is a finite complex-valued Borel measure supported on $[c;d] \subset (0; \infty)$, satisfies
	\[ \liminf_{r \to \infty} \frac{\mu(\{ x \in [0; r] \mid f(x) \geq 0 \})}{r} \geq \frac{c}{c+d}, \]
where the lower bound is sharp. We also note paper \cite{KozmaOravecz}, in which G. Kozma and F. Oravecz established that the supremum, over non-zero polynomials of the form (\ref{form_m_n}), of the lengths of zero-free intervals is
	\[ 2\pi\frac{n - m + 1}{n + m}. \]
	
	Let $D$ be a given set of positive integers. Following \cite{Pandey} we call $A \subset \mathbb{N}$ a $D$-set if $a_1, a_2 \in A \Rightarrow a_1 - a_2 \notin D$. The maximal density of a $D$-set
\[ \alpha(D)=\limsup_{r\to\infty}\max_{\substack{A\subset\{1,\ldots,r\}\\ A\text{ is a }D\text{-set}}}\frac{|A|}{r},\qquad r\in\mathbb N. \]
	is called the Motzkin density. The problem of determining this quantity originated from an unpublished collection of problems posed by T. S. Motzkin (see \cite{CantorGordon}). For example, for $D = \{1, ... n\}$ one has the Motzkin density $\frac{1}{n+1}$: just pick numbers divisible by $n+1$ as a $D$-set.

	One of the main theorems in this field is
	\begin{theorem}
		\label{cvet_bound_initial}
		Let $f$ be a non-zero trigonometric polynomial
		\[ f(x) = \sum_{d \in D} [a_d \cos(dx) + b_d \sin(dx)], \]
		where $D$ is a finite non-empty set of natural numbers, then the following lower bound
		\begin{equation}
			\label{our_bound} 
			\alpha(D) \leq \gamma_{+}(f)
		\end{equation}
		holds true.
	\end{theorem}
	For $D = \{m, ..., n\}$ this bound coincides with the Yudin's one.

	In the area of combinatorics inequality (\ref{our_bound}) is well-known: picking an appropriate $f$ one gets an upper bound on $\alpha(D)$ (see, for example, \cite{MatolcsiRuzsa} or Section 3 in \cite{XuYip}). Though usually one refers to a more general Cvetkovic bound (also known as the inertia bound): an upper bound on the independence number of a graph in terms of the number of positive and negative eigenvalues of its weighted adjacency matrix (see Section \ref{prelim_cvet}).

	In our paper we consider extensions of Taikov's problem in the following directions. First, for completeness, in Section \ref{conn_main} we give a proof of Theorem \ref{cvet_bound_initial}, which was not considered in \cite{Babenko}, \cite{Yudin}. Second, in Section \ref{multivariable_section} we generalize it to many variables (see Theorem \ref{m_connect}). Instead of a trigonometric polynomial we take a Laurent polynomial in several variables
	\[ f(z) = f(z_1, \ldots, z_m) = \sum_{d \in D} c_d z_1^{d_1}\ldots z_m^{d_m} \]
	and consider the measure of the positivity set
	\[ \{ z \in \mathbb{T}^m \mid \operatorname{Re} f(z) > 0 \}, \]
	where $\mathbb{T}$ is the unit torus $\mathbb{T} = \{ z \mid |z| = 1 \}$. In Section \ref{ap_section} we apply the multivariable version of our lower bound and show that inequality (\ref{our_bound}) holds for almost periodic functions with finite spectrum. The main result of this work consists in the following theorem
	\begin{customthm}{5.2}
		Let $f$ be a non-zero continuous almost periodic function of the form
\[ f(x)=\sum_{v=1}^{\infty}[a_v\cos(\theta_v x)+b_v\sin(\theta_v x)],\qquad a_v,b_v\in\mathbb R,\quad\theta_v>0. \]
		Assume that the series uniformly converges to $f$ and that quantity
\begin{equation}\label{m_e}m(\varepsilon)=\limsup_{A\to\infty}\frac{\mu(\{x\in[-A;A]:|f(x)|\leq\varepsilon\})}{2A}.\end{equation}
		tends to zero as $\varepsilon$ tends to zero. Also let $\{\theta_v\}$ be contained in a union of segments
		\[ \{\theta_v\} \subset [c_1; d_1] \cup \ldots \cup [c_m; d_m], \quad 0 < c_1 < d_1 < \ldots < c_m < d_m, \]
		then $\gamma_+(f)$ can be bounded from below as follows
\begin{equation*}\gamma_+(f)=\liminf_{r\to\infty}\frac{\mu(\{x\in[-r;r]:f(x)>0\})}{2r}\geq\prod_{k=1}^{m}\frac{c_k}{c_k+d_k}.\end{equation*}
	\end{customthm}

	\newpage

	\section{Preliminaries}
	\label{prelim}
	Call a subset $S \subseteq V$ of vertices of a graph $G = (V; E)$ independent if there are no edges between vertices from $S$
\[ i,j\in S\ \Longrightarrow\ (i,j)\notin E. \]
	By $\alpha(G)$ denote the size of the largest independent subset.

	\subsection{Motzkin density}
	\label{prelim_motzkin} 
	With a finite set of natural numbers $D$ we associate a family of graphs $G(n, D)$. Graph $G(n, D)$ has $n$ vertices labeled $0, \ldots, n-1$ and an edge is drawn between a pair of vertices $i, j$ whenever they are at a distance $d \in D$, in other words,
	\begin{equation*} (i, j) \in E(n, D) \Leftrightarrow i - j \equiv d \Mod{n} \textrm{ or } j - i \equiv d \Mod{n} \textrm{ for some } d \in D. \end{equation*}

	Note that graphs $G(n, D)$ can have loops for $n \leq \max(D)$.

	Graphs $G(n, D)$ are called circulant. Their <<homogeneity>> suggests that the independence number $\alpha(G(n, D))$ should be <<proportional>> to $n$. And, indeed, by Lemma \ref{alpha_limit}
	\[ \lim_{n \to \infty} \frac{\alpha(G(n, D))}{n} = \sup_{n} \frac{\alpha(G(n, D))}{n}. \]
	We denote this limit by $\alpha(D)$. Quantity $\alpha(D)$ is known as the Motzkin density or Ruzsa's measure of intersectivity and has several other similar definitions \cite{Pandey}. To see that this agrees with the definition in the introduction, let $b_r$ be the maximum size of a $D$-set in $\{0,\ldots,r-1\}$ and put $R=\max D$. Every independent set of $G(r,D)$ is such a $D$-set, while every such $D$-set embeds in $G(r+R,D)$ without new edges. Hence
\[ \alpha(G(r,D))\leq b_r\leq\alpha(G(r+R,D)), \]
and division by $r$ and Lemma \ref{alpha_limit} give the same limit.

	To gain a sense of $\alpha(D)$ consider case $|D| \leq 2$ first. This case was examined by Cantor and Gordon in \cite{CantorGordon}. A one-element set has Motzkin density of one half
	\[ \alpha(\{x\}) = \frac{1}{2}. \]
	For a pair of distances $D = \{a, b\}$ the authors first observe that for natural $d$
	\[ \alpha(\{m_1, m_2, \ldots\}) = \alpha(\{d m_1, d m_2, \ldots \}), \]
	meaning that the problem can be reduced to coprime $a$ and $b$, for which they later prove
	\begin{equation}\label{alpha_a_b} \alpha(\{a, b\}) = \frac{\lfloor \frac{a+b}{2} \rfloor}{a + b}. \end{equation}

	For an arithmetic progression $D = \{a, a + d, \ldots, a + (n-1)d\}$ with $\gcd(a,d)=1$ (after division by the common gcd if necessary) Gupta and Tripathi \cite{GuptaTripathi} have shown
	\begin{equation} \alpha(D) = \begin{cases} \frac{2a + (n-1)(d-1)}{2\{ 2a + (n-1) d \}} & \textrm{if $d$ is odd,} \\ \frac{1}{2} & \textrm{if $d$ is even.}\end{cases}\end{equation}
	In particular, for consecutive integers $D = \{a, a + 1, \ldots, b\}$ holds equality
	\begin{equation}\label{alpha_consec} \alpha(\{a, a + 1, \ldots, b\}) = \frac{a}{a + b}. \end{equation}
	Setting $a = 1$ we obtain
	\begin{equation}\label{alpha_first} \alpha(\{1, \ldots, n\}) = \frac{1}{n+1}. \end{equation}

	In Proposition \ref{union} we observe that the Motzkin density of a union $\alpha(D_1 \cup D_2)$ can be bounded from below by the product $\alpha(D_1) \alpha(D_2)$. We combine this observation with (\ref{alpha_consec}) and derive a lower bound
	\bea
	\label{alpha_segments}
	\nonumber && \alpha(\{ a_1, a_1 + 1, \ldots, b_1, a_2, a_2 + 1, \ldots, b_2, \ldots, a_m, a_m + 1, \ldots, b_m\}) \\
	&& \makebox[6em]{} \geq \frac{a_1}{a_1 + b_1} \times \ldots \times \frac{a_m}{a_m + b_m}.
	\eea
	In other words, for a set of differences $D$ contained in a union of segments
	\[ [c_1; d_1] \cup \ldots \cup [c_m; d_m] \subset (0; \infty) \]
	holds the following inequality
	\begin{equation}\label{alpha_union} \alpha(D) \geq \frac{c_1}{c_1 + d_1} \times \ldots \times \frac{c_m}{c_m + d_m}. \end{equation}
Indeed, discard empty integer intersections and put $a_j=\lceil c_j\rceil$, $b_j=\lfloor d_j\rfloor$ in each remaining segment. Then $a_j/(a_j+b_j)\geq c_j/(c_j+d_j)$; monotonicity in $D$ and the union bound give (\ref{alpha_union}).
	\subsection{Taikov's problem}
	\label{prelim_taikov}
	We represent a trigonometric polynomial $f$ as
\[ f(x)=\operatorname{Re}p(e^{ix}). \]
	where $p$ is a complex polynomial
	\[ p(z) = \sum_{k = 1}^{n} (a_k - i b_k) z^k. \]

	For a finite set of natural numbers $D$ let $P_{D}$ be the set of non-zero complex polynomials of the form
	\[ p(z) = \sum_{d \in D} c_d z^d. \]
	It can be shown that the quantity
\begin{equation}\gamma_+(p)=\frac{1}{2\pi}\mu(\{x\in[0;2\pi]:\operatorname{Re}p(e^{ix})>0\})\end{equation}
	attains its minimum value on $P_{D}$. We choose to ignore this observation and instead define
	\begin{equation} \gamma(D) = \inf_{p \in P_{D}} \gamma_{+}(p). \end{equation}

	We omit an analogous quantity
\[ \gamma_-(p)=\frac{1}{2\pi}\mu(\{x\in[0;2\pi]:\operatorname{Re}p(e^{ix})<0\}). \]
	in the above definition due to relation
	\[ \gamma_{-}(p) = \gamma_{+}(-p). \]

Unless explicitly stated otherwise, sets of differences are non-empty. For the empty set we use the convention $\alpha(\varnothing)=1$.

	\subsection{Cvetkovic bound}
	\label{prelim_cvet}
	Let $G$ be a graph with $n$ vertices labeled $1, \ldots, n$. Call an $n \times n$ matrix $M$ a $G$-matrix if its non-zero entries correspond to graphs edges
	\[ m_{ij} \neq 0 \Rightarrow (i, j) \in E. \]
	\begin{theorem}[Cvetkovic bound \cite{Bollobas}]
		\label{cvetkovic_bound}
		A hermitian $G$-matrix $M$ has at least $\alpha(G)$ non-negative and at least $\alpha(G)$ non-positive eigenvalues, counted with multiplicity. 
	\end{theorem}
	\begin{proof}
		Consider a hermitian form
		\[ \varphi(\vec{x}, \vec{y}) = \sum_{i = 1}^{n} \sum_{j = 1}^{n} m_{ij} x_i \overline{y_j}. \]
		Counted with multiplicity the number of positive eigenvalues of $M$ is equal to the positive index of inertia $n_+$ of $\varphi$, the number of negative eigenvalues -- to the negative index of inertia $n_{-}$.

		We are going to show that there are at least $\alpha(G)$ non-positive eigenvalues
		\begin{equation}\label{eq_1} \alpha(G) \leq n - n_{+}. \end{equation}
		Let $W_{+}$ be an $n_{+}$-dimensional subspace of $\mathbb{C}^n$ such that the restriction of $\varphi$ to $W_{+}$ is a positive-definite form. Take an independent subset $S \subseteq V$ of size $\alpha(G)$. Denote by $U$ the corresponding subspace of $\mathbb{C}^n$
		\[ U = \{ \vec{x} \mid \forall i \notin S\colon x_i = 0 \}. \]
		Since $m_{ij} = 0$ for $i, j \in S$, the restriction of $\varphi$ to $U$ is identically zero.

		Subspaces $W_{+}$ and $U$ have no non-zero vectors in common due to
		\[ \forall \vec{x} \in W_{+}\setminus\{0\}\colon \ \varphi(\vec{x}, \vec{x}) > 0, \quad \forall \vec{x} \in U\colon \ \varphi(\vec{x}, \vec{x}) = 0. \]
		Dimensions of such subspaces should satisfy an inequality
		\[ \dim W_{+} + \dim U \leq n, \]
		\[ n_{+} + \alpha(G) \leq n, \]
		from which our claim (\ref{eq_1}) follows.

		Proof of our bound for the number of non-negative eigenvalues is analogous.
	\end{proof}

	\section{Connection between the two problems}
	\label{conn_main}

	Cvetkovic bound applied to circulant graphs $G(n, D)$ results in the following theorem.
	\begin{theorem}
		\label{main_theorem}
		Inequality
		\[ \alpha(D) \leq \gamma(D) \]
		holds for finite sets of natural numbers $D$.
	\end{theorem}
	\begin{proof}
		It would suffice to show that inequality
		\begin{equation}\label{eq_2} \alpha(D) \leq \gamma_{+}(p) \end{equation}
		holds for every non-zero complex polynomial of the form
		\[ p(z) = \sum_{d \in D} c_d z^d. \]

		Let $c_i$ be equal to zero for $i \notin D$. For $n \geq 2\max(D) + 1$ we construct an $n \times n$ hermitian circulant matrix
		\begin{equation*}
			M = (m_{ij}) = \begin{bmatrix}
				0 & c_1 & \ldots & c_{\max(D)} & \ldots & \overline{c_{\max(D)}} & \ldots & \overline{c_1} \\
				\overline{c_1} & 0 & c_1 & \ldots & c_{\max(D)} & \ldots & \overline{c_{\max(D)}} & \ldots \\
				\vdots & \vdots & \vdots & \vdots & \vdots & \vdots & \vdots & \vdots \\
				\vdots & \vdots & \vdots & \vdots & \vdots & \vdots & \vdots & \vdots \\
				\vdots & \vdots & \vdots & \vdots & \vdots & \vdots & \vdots & \vdots \\
				\vdots & \vdots & \vdots & \vdots & \vdots & \vdots & \vdots & \vdots \\
				\vdots & \vdots & \vdots & \vdots & \vdots & \vdots & \vdots & \vdots \\
				c_1 & \ldots & c_{\max(D)} & \ldots & \overline{c_{\max(D)}} & \ldots & \overline{c_1} & 0
			\end{bmatrix},
		\end{equation*}
		whose first row begins with a zero followed by coefficients $c_k$ of $p$ up to $c_{\max(D)}$ and ends with their conjugates $\overline{c_k}$ in reverse order, entries between $c_{\max(D)}$ and $\overline{c_{\max(D)}}$ are zeroes. Each next row is a circular shift to the right of a previous one.

		Denote an $n$-th root of unity $e^{\frac{2\pi i}{n}}$ by $\omega_n$. We define vectors
		\[
			\vec{v}_j = \begin{bmatrix}
				1 \\ \omega_n^{j} \\ \omega_n^{2j} \\ \vdots \\ \omega_n^{(n - 1)j}
			\end{bmatrix}
		\]
		for $j = 0, \ldots, n - 1$. These vectors are eigenvectors of our matrix $M$
		\[ M \vec{v}_j = \lambda_j \vec{v}_j, \]
		where
		\[
			\lambda_j = \sum_{d \in D} [ c_d (\omega_n^j)^d + \overline{c_d} (\omega_n^j)^{n - d} ] = p(\omega_n^j) + \overline{p(\omega_n^j)} = 2 \Rea p(\omega_n^j).
		\]
		Moreover, vectors $\vec{v}_j$ form a basis of $\mathbb{C}^n$. We conclude that numbers
		\[ 2 \Rea p(1), 2 \Rea p(e^{2 \pi i \frac{1}{n}}), \ldots, 2 \Rea p(e^{2 \pi i \frac{n - 1}{n}}) \]
		are the eigenvalues of $M$.

		For a circulant graph $G = G(n, D)$ matrix $M$ is a $G$-matrix, which means we can apply Theorem \ref{cvetkovic_bound}
\[ \alpha(G(n,D))\leq |\{j\in\{0,\ldots,n-1\}:2\operatorname{Re}p(e^{2\pi i j/n})\geq0\}|. \]
		Divide both sides by $n$
\[ \frac{\alpha(G(n,D))}{n}\leq\frac{|\{j\in\{0,\ldots,n-1\}:2\operatorname{Re}p(e^{2\pi i j/n})\geq0\}|}{n}. \]
	\end{proof}

	To give a simple example of how Theorem \ref{main_theorem} can be applied, take a non-zero function
	\[ f(x) = A \cos(px) + B \cos(qx) \]
	with $p$ and $q$ coprime and $p + q$ odd. Combining Theorem \ref{main_theorem} with inequality (\ref{alpha_a_b}) we get a lower bound on the measure of the positivity set
	\[ \frac{\mu(\{ x \in [0; 2\pi] \mid f(x) > 0 \})}{2 \pi} \geq \frac{1}{2} - \frac{1}{2(p + q)}. \]

	Combining Theorem \ref{main_theorem} with inequality (\ref{alpha_union}) we obtain the following result.
	\begin{theorem}
		\label{theorem_rational}
		Let $f$ be a non-zero function with period $T$ of the form
		\[ f(x) = \sum_{k = 1}^{n} [a_k \cos(\theta_k x) + b_k \sin(\theta_k x)], \]
		where $a_k,b_k\in\mathbb R$ and $\theta_k$ are rational numbers contained in a union of segments
		\[ \{ \theta_1, \ldots, \theta_n \} \subset [c_1; d_1] \cup \ldots \cup [c_m; d_m] \subset (0; \infty). \] The measure of the positivity set can be bounded from below by
		\begin{equation}\label{eq_rational} \frac{\mu(\{ x \in [0; T] \mid f(x) > 0 \})}{T} \geq \frac{1}{1 + \frac{d_1}{c_1}} \times \ldots \times \frac{1}{1 + \frac{d_m}{c_m}}. \end{equation}
	\end{theorem}
	\begin{proof}
Take $q\in\mathbb N$ to be a common denominator of the $\theta_k$ and replace $f(x)$ by $g(x)=f(qx)$. The frequencies $q\theta_k$ of $g$ are now positive integers contained in a union of segments
		\[ [q c_1; q d_1] \cup \ldots \cup [q c_m; q d_m]. \]
Both the normalized positivity measure and the right hand side of (\ref{eq_rational}) are unchanged by this scaling. The normalized positivity measure is the same for any period of a non-constant continuous periodic function. Applying Theorem \ref{main_theorem} to the $2\pi$-periodic function $g$ and using (\ref{alpha_union}) proves the claim.
	\end{proof}

	\section{Polynomials in several variables}
	\label{multivariable_section}

	The goal of this section is to extend our results to functions of several variables like
	\[ f(x, y, \ldots) = \cos(x + y) + \sin(3x)\cos(4y) + \ldots \]
	It would be convenient for us to represent these functions as
	\[ f(x, y, \ldots) = p(e^{ix}, e^{iy}, \ldots), \]
	where $p$ is a Laurent multivariable polynomial
	\[ p(z_1, \ldots, z_m) = \sum_{d} c_d z_1^{d_1} \ldots z_m^{d_m}. \]
	Here we index the coefficients $c_d$ by vectors $d = (d_1, \ldots, d_m)$ from $\mathbb{Z}^m$. Let $-d$ be the opposite vector $(-d_1, \ldots, -d_m)$. Note that for $f$ to be a real-valued function we should also require
	\begin{equation}\label{conjugates} c_{-d} = \overline{c_d}. \end{equation}

	Arguments $e^{ix}$ sit on the unit circle in the complex plane
	\[ \mathbb{T} = \{ z \mid |z| = 1 \}. \]
	By $\lambda_m$ we denote the normalized Lebesgue measure on $\mathbb{T}^m$. In other words, the Lebesgue measure on $\mathbb{T}^m$ could be multiplied by a constant so that $\lambda_m(\mathbb{T}^m) = 1$.

	We are now ready to formulate the generalization of Taikov's problem. As the set of differences $D$ we take a finite symmetric set of non-zero integer vectors, i.e.,
	\begin{equation}\label{set_D} D \subset \mathbb{Z}^m, \quad (0, \ldots, 0) \notin D, \quad d \in D \Leftrightarrow -d \in D. \end{equation}
	Let $P_{D}$ be the set of non-zero complex Laurent polynomials of the form
	\begin{equation}\label{lemma_a_2_poly} p(z_1, \ldots, z_m) = \sum_{d \in D} c_d z_1^{d_1}\ldots z_m^{d_m} \end{equation}
	with an additional condition (\ref{conjugates}). On the torus $\mathbb{T}^m$ such functions assume real values
	\[ p\colon \mathbb{T}^m \to \mathbb{R}. \]
	Over this set $P_{D}$ we seek to minimize the quantity
	\begin{equation} \gamma_{+}(p) = \lambda_m(\{ z \in \mathbb{T}^m \mid p(z) > 0 \}). \end{equation}
	It can be shown that $\gamma_{+}(p)$ attains its minimum value. We choose to ignore this observation and instead define
	\[ \gamma(D) = \inf_{p \in P_{D}} \gamma_{+}(p). \]

	In order to reproduce our proof we would need to replace circulant graphs with their multidimensional analogue. Roughly speaking, it would be like replacing integers on a line with a grid.

	Formally multidimensional circulant graphs can be defined as follows. Let $D$ be a finite set of integer vectors of the form (\ref{set_D}). In other words, $D$ consists of pairs of opposite non-zero vectors. We call $D$ the set of differences. With each such set of vectors we associate a family of graphs $G(n, D)$. Each $G(n, D)$ has $\{0, \ldots, n - 1\}^m$ as its set of vertices and two vertices $u$ and $v$ are adjacent if there is a vector $d \in D$ such that $u - v$ is congruent to $d$ modulo $n$
	\[ (u, v) \in E(n, D) \Leftrightarrow \exists d \in D\colon \forall i\colon u_i - v_i \equiv d_i \Mod{n}.\]

	Note that graphs $G(n, D)$ can have loops for $n$ not greater than $\max_{d \in D} (\max_{i} d_i)$.

	Again we can show that the independence number $\alpha(G(n, D))$ is <<proportional>> to the size of the graph $G(n, D)$. By Lemma \ref{m_alpha_limit}
	\begin{equation}
		\label{m_limit}
		\lim_{n \to \infty} \frac{\alpha(G(n, D))}{n^m} = \sup_{n} \frac{\alpha(G(n, D))}{n^m}.
	\end{equation}
	We denote this limit by $\alpha(D)$.

	In the proof of Theorem \ref{main_theorem} we related Taikov's problem to Motzkin density by approximating the set $\{ z \mid p(z) > 0 \}$ with roots of unity distributed on the circle. Here we encounter a measure-theoretic difficulty: can we approximate domains in $\mathbb{T}^m$ like this?

	Denote the exponential function by $\exp$. By the $n$-th lattice we mean the set of points on $\mathbb{T}^m$ defined as
	\[
		B_n = \left\{ \left.\left(\exp\left(\frac{2\pi i}{2^n} a_1\right), \ldots, \exp\left(\frac{2\pi i}{2^n} a_m\right)\right) \right| (a_1, \ldots, a_m) \in \mathbb{Z}^m \right\}.
	\]
	By the $n$-th closed box centered at a point $(x_1, \ldots, x_m) \in B_n$ we mean the set
	\[
		\left\{ \left(x_1 \exp(2\pi i \theta_1), \ldots, x_m \exp(2\pi i \theta_m) \right) \left| (\theta_1, \ldots, \theta_m) \in \left[-\frac{1}{2^{n + 1}}; +\frac{1}{2^{n + 1}} \right]^m \right.\right\}.
	\]
	Note that the $n$-th closed boxes cover $\mathbb{T}^m$.

	For a function $p$ discussed above define sets
	\begin{align*}
		F_+ = \{ z \in \mathbb{T}^m \mid p(z) > 0 \},\\
		F_- = \{ z \in \mathbb{T}^m \mid p(z) < 0 \},\\
		F_0 = \{ z \in \mathbb{T}^m \mid p(z) = 0 \}.
	\end{align*}
	According to Lemma \ref{zero_set} the zero set $F_0$ is of zero measure
	\begin{equation}\label{zero_set_eq} \lambda_m(F_+) = \gamma_{+}(p), \quad \lambda_m(F_-) = 1 - \gamma_{+}(p), \quad \lambda_m(F_0) = 0. \end{equation}

	Finally, in Lemma \ref{dyadic_approx} we show that sets $F_{+} \cap B_n, F_{-} \cap B_n, F_0 \cap B_n$ approximate $F_+, F_-, F_0$, respectively:
	\begin{align}
		\label{approximation_a}
		\lim_{n \to \infty} \frac{|F_+ \cap B_n|}{2^{nm}} &= \lambda_m(F_+), \\
		\label{approximation_b}
		\lim_{n \to \infty} \frac{|F_- \cap B_n|}{2^{nm}} &= \lambda_m(F_-), \\
		\label{approximation_c}
		\lim_{n \to \infty} \frac{|F_0 \cap B_n|}{2^{nm}} &= 0.
	\end{align}

	Let us remind the reader of how the multidimensional Fourier transform works.

	We can add, subtract and multiply by an integer vectors from $\mathbb{Z}_n^m$. Denote the zero vector by $\bm{0}$. We index matrices by the set $\mathbb{Z}_n^m$. Call a matrix $A = (a_{\bm{i} \bm{j}})$ $m$-circulant if the value of the entry $a_{\bm{i} \bm{j}}$ is determined by the difference $\bm{i} - \bm{j}$, i.e.,
	\[ \bm{i} - \bm{j} = \bm{k} - \bm{l} \Rightarrow a_{\bm{i} \bm{j}} = a_{\bm{k} \bm{l}}. \]

	Let $\omega_n$ be an $n$-th root of unity \[\omega_n = e^{\frac{2\pi i}{n}}.\]

	\begin{lemma}
		\label{m_eigen}
		Let $A = (a_{\bm{i} \bm{j}})$ be an $m$-circulant matrix. Denote entries of the row corresponding to $\bm{0}$ by $x_{\bm{i}} = a_{\bm{0} \bm{i}}$.
		Then counted with multiplicity the numbers
		\[ \lambda_{\bm{j}} = \sum_{\bm{i}} x_{\bm{i}} \omega_n^{i_1 j_1} \ldots \omega_n^{i_m j_m}, \quad \bm{j} \in \mathbb{Z}_n^m, \]
		are the eigenvalues of $A$.
	\end{lemma}
	\begin{proof}
		For an index $\bm{k} \in \mathbb{Z}_{n}^m$ consider a vector $w^{(\bm{k})}$
		\[
			w^{(\bm{k})}_{\bm{j}} = \omega_{n}^{k_1 j_1} \ldots \omega_n^{k_m j_m}.
		\]
		Note that $w^{(\bm{k})}$ is an eigenvector of the matrix $A$ with eigenvalue $\lambda_{\bm{k}}$, since
\begin{align*}
(Aw^{(\bm k)})_{\bm i}
 &=\sum_{\bm j}a_{\bm i\bm j}w^{(\bm k)}_{\bm j}
 =\sum_{\bm j}x_{\bm j-\bm i}\prod_{r=1}^m\omega_n^{k_rj_r}\\
 &=\left(\sum_{\bm h}x_{\bm h}\prod_{r=1}^m\omega_n^{k_rh_r}\right)
   \prod_{r=1}^m\omega_n^{k_ri_r}
 =\lambda_{\bm k}w^{(\bm k)}_{\bm i}.
\end{align*}

		Also vectors $\{ w^{(\bm{k})} \}_{\bm{k} \in \mathbb{Z}_n^m}$ form a basis. Indeed, if we regard our vector space as a hermitian space with hermitian form
		\[
			(u, v) = \sum_{\bm{i}} u_{\bm{i}} \overline{v_{\bm{i}}},
		\]
		then $w^{(\bm{i})}$ and $w^{(\bm{j})}$ shall be orthogonal whenever $\bm{i} \neq \bm{j}$
\begin{align*}
(w^{(\bm i)},w^{(\bm j)})
 &=\sum_{\bm k}\prod_{r=1}^m\omega_n^{k_r(i_r-j_r)}\\
 &=\prod_{r=1}^m\left(\sum_{k_r=0}^{n-1}\omega_n^{k_r(i_r-j_r)}\right)=0
 \qquad(\bm i\ne\bm j).
\end{align*}

		So we have found a basis of eigenvectors, whose eigenvalues are $\lambda_{\bm{j}}$.
	\end{proof}

	Now we are ready to repeat the proof of Theorem \ref{main_theorem}.
	\begin{theorem}
		\label{m_connect}
		An inequality
		\[
			\alpha(D) \leq \gamma(D)
		\]
		holds for a finite set of differences (\ref{set_D}).
	\end{theorem}
	\begin{proof}
		It would suffice to show that the inequality
		\begin{equation}\label{eq_4} \alpha(D) \leq \gamma_{+}(p) \end{equation}
		holds for every non-zero polynomial
		\[ p(z) = \sum_{d \in D} c_d z_1^{d_1} \ldots z_m^{d_m} \]
		with $c_{-d} = \overline{c_d}$.

		For $n \geq 1 + 2\max_{d \in D} (\max_{i} d_i)$ we construct a hermitian $m$-circulant matrix $A$: let $a_{\bm{i} \bm{j}} = c_{d}$ if $d \in D$ is congruent to $\bm{j} - \bm{i}$, and let other entries be zero.
		By Lemma \ref{m_eigen} the eigenvalues of $A$ are the numbers
		\[
\lambda_{\bm j}=\sum_{d\in D}c_d\omega_n^{j_1d_1}\cdots\omega_n^{j_md_m}=p(\omega_n^{j_1},\ldots,\omega_n^{j_m}).
		\]

		For a multidimensional circulant graph $G = G(n, D)$ matrix $A$ is a $G$-matrix, which means we can apply Theorem \ref{cvetkovic_bound}
\[ \alpha(G(n,D))\leq|\{\bm j\in\mathbb Z_n^m:\lambda_{\bm j}\geq0\}|. \]
		Divide both sides by $n^m$
\[ \frac{\alpha(G(n,D))}{n^m}\leq\frac{|\{\bm j\in\mathbb Z_n^m:\lambda_{\bm j}\geq0\}|}{n^m}. \]
		Choose $n$ to be a power of two $2^k$ tending to infinity. By (\ref{m_limit}) the left hand side tends to $\alpha(D)$. The right hand side is equal
		\[ \frac{|F_+ \cap B_k|}{2^{km}} + \frac{|F_0 \cap B_k|}{2^{km}} \]
		and by (\ref{zero_set_eq}), (\ref{approximation_a}), (\ref{approximation_c}) tends to
		\[ \lambda_m(F_+) = \gamma_+(p). \]
		In limit we get the desired inequality (\ref{eq_4}).
	\end{proof}

	\section{Almost periodic functions}
	\label{ap_section}

	In this section we present an argument that extends Theorem \ref{main_theorem} to the case of almost periodic functions. Consider, for example,
	\[ f(x) = \cos(\alpha x) + \cos(\beta x), \]
	where $\alpha$ and $\beta$ are incommensurable. By periodicity $f$ only depends on
	\[ \alpha x \mod 2\pi \textrm{ and } \beta x \mod 2\pi. \]
	Such a situation is called an irrational winding of the torus. For our problem this means that we can think of arguments $\alpha x$ and $\beta x$ as independent variables. It then becomes clear that
	\[ \lim_{r \to \infty} \frac{\mu(\{ x \in [0; r] | f(x) > 0 \})}{r} = \frac{1}{2}. \]

	Functions like $f$ are called almost periodic. The classical argument presented above shows that these functions can behave like trigonometric polynomials of several variables. But what if apart from $\alpha$ and $\beta$ we have some other $\gamma$ and the three are subject to some relation over $\mathbb{Z}$? In Theorem \ref{ap_theorem} we are going to resolve this issue and reduce the problem of almost periodic functions to trigonometric polynomials of several variables. To the latter the results of the previous section can be applied.

	We first extend Motzkin density to arbitrary sets of real numbers.

	Let $D$ be a bounded subset of $\mathbb{R}_+=(0,\infty)$. Call a subset $A \subset \mathbb{R}$ a $D$-set if there are no $x, y \in A$ whose difference $x - y$ belongs to $D$. For a positive $l$ by $\alpha(l, D)$ we mean the supremum of all possible measures that open $D$-set contained in interval $(0; l)$ could have
	\[ \alpha(l, D) = \sup_{\textrm{open $D$-set } A \subset (0; l)} \mu(A). \]

	By Lemma \ref{ap_alpha_limit} we have
	\begin{equation}\label{ap_alpha_def} \lim_{l \to \infty} \frac{\alpha(l, D)}{l} = \limsup_{l \to \infty} \frac{\alpha(l, D)}{l}. \end{equation}
	This limit we take for the definition of $\alpha(D)$.

	For an unbounded set of differences $D$ let $\alpha(D)$ be equal to
	\[ \alpha(D) = \inf_{L > 0} \alpha(D \cap (0; L)). \]
The quantities $\alpha(D\cap(0;L))$ are non-increasing in $L$, with $\alpha(\varnothing)=1$. For bounded $D$ they equal $\alpha(D)$ once $L>\sup D$, so the definitions agree. We also use monotonicity: $D_1\subseteq D_2$ implies $\alpha(D_1)\geq\alpha(D_2)$.

	Now assume that $D$ is finite
	\begin{equation}\label{fin_set_D} D = \{ \alpha_1, \ldots, \alpha_m \}. \end{equation}
	In a $\mathbb{Z}$-module
	\[ \alpha_1 \mathbb{Z} + \ldots + \alpha_m \mathbb{Z} \]
we take a basis $a_1,\ldots,a_n$ of positive real numbers. Such a basis exists because a finitely generated torsion-free abelian group is free; the signs of its basis elements can be chosen positive. This means that $a_i$ can be expressed as
	\[ a_i = c_{i1} \alpha_1 + \ldots + c_{im} \alpha_m \textrm{ with } c_{ij} \in \mathbb{Z}, \]
	that $a_i$ are $\mathbb{Z}$-linearly independent, and that $\alpha_i$ can be written as
	\begin{equation}\label{vec_d_j} \alpha_j = d_{j1} a_1 + \ldots + d_{jn} a_n \textrm{ with } d_{ji} \in \mathbb{Z}. \end{equation}

	Turns out the Motzkin density of $D$ can be related to the Motzkin density of integer vectors $\{ \pm \vec{d}_j \}$: by Lemma \ref{alpha_r_T}
	\begin{equation}\label{alpha_r_T_eq} \alpha(D) \leq \alpha(\{ \pm \vec{d}_j \}). \end{equation}

	We are now ready to generalize Theorem \ref{main_theorem} to almost periodic functions.
	\begin{theorem}
		\label{ap_theorem}
		Let $f$ be a non-zero continuous almost periodic function of the form
\[ f(x)=\sum_{v=1}^{\infty}[a_v\cos(\theta_v x)+b_v\sin(\theta_v x)],\qquad a_v,b_v\in\mathbb R,\quad\theta_v>0. \]
		Assume that the series uniformly converges to $f$ and that quantity
\begin{equation}\label{m_ee}m(\varepsilon)=\limsup_{A\to\infty}\frac{\mu(\{x\in[-A;A]:|f(x)|\leq\varepsilon\})}{2A}.\end{equation}
		tends to zero as $\varepsilon$ tends to zero. Then we have
\begin{equation}\label{ap_equation}\gamma_+(f):=\liminf_{r\to\infty}\frac{\mu(\{x\in[-r;r]:f(x)>0\})}{2r}\geq\alpha(\{\theta_v\}).\end{equation}
	\end{theorem}
	\begin{proof}
For an arbitrary $\delta>0$, choose $0<\varepsilon<M/2$, where $M$ is the supremum of the absolute value of $f$, such that $m(\varepsilon)<\delta$. By uniform convergence choose $m$ with
\[ f_m(x)=\sum_{v=1}^m[a_v\cos(\theta_vx)+b_v\sin(\theta_vx)],\qquad
   \forall x : |f(x) - f_m(x)| <\varepsilon. \]
Then $f_m\not\equiv0$. The indicators of $\{f>0\}$ and $\{f_m>0\}$ can differ only on $\{|f|\leq\varepsilon\}$. Integrating this pointwise bound on $[-A,A]$, dividing by $2A$, we obtain
\[ \gamma_+(f)\geq\gamma_+(f_m)-m(\varepsilon)\geq\gamma_+(f_m)-\delta. \]

Put $D_m=\{\theta_1,\ldots,\theta_m\}$. Choose a basis of positive numbers $\beta_1,\ldots,\beta_n$ in the $\mathbb Z$-module generated by $D_m$. Thus the $\beta_i$ are $\mathbb Z$-linearly independent and
\[ \theta_v=d_{v1}\beta_1+\cdots+d_{vn}\beta_n,\qquad d_{vi}\in\mathbb Z. \]
Write $D'_m=\{\pm(d_{v1},\ldots,d_{vn}):1\leq v\leq m\}$. Since $\theta_v>0$, this finite symmetric set contains no zero vector.

		One could rewrite $a_v \cos(\theta_v x) + b_v \sin(\theta_v x)$ as
		\[ \frac{1}{2}(a_v - ib_v) e^{i\theta_v x} + \frac{1}{2} (a_v + ib_v) e^{-i\theta_v x}. \]
		Denote $e^{i\beta_jx}$ by $z_j$, then we can represent functions $e^{\pm i \theta_v x}$ as
		\[ e^{\pm i \theta_v x} = z_1^{\pm d_{v1}} \ldots z_n^{\pm d_{vn}}. \]
		So for a Laurent polynomial
\[
F_m(z_1,\ldots,z_n)=\frac12\sum_{v=1}^m\left[
(a_v-ib_v)\prod_{j=1}^nz_j^{d_{vj}}+
(a_v+ib_v)\prod_{j=1}^nz_j^{-d_{vj}}\right].
\]
		we have
\[ f_m(x)=F_m(e^{i\beta_1x},\ldots,e^{i\beta_nx})=\sum_{v=1}^{m}[a_v\cos(\theta_vx)+b_v\sin(\theta_vx)]. \]

		By Kronecker--Weyl equidistribution theorem
		\[ \gamma_+(f_m) = \gamma_+(F_m), \]
where $\gamma_+(F_m)$ is defined using normalized measure $\lambda_n$ on $\mathbb T^n$. Indeed, let $\sigma(x)=(e^{i\beta_1x},\ldots,e^{i\beta_nx})$. By the $\mathbb Z$-linear independence of the basis $\beta_j$, the continuous Kronecker--Weyl theorem gives, for every set $U\subset\mathbb T^n$ whose boundary has measure zero,
\[ \lim_{A\to\infty}\frac{\mu(\sigma^{-1}(U)\cap[-A;A])}{2A}=\lambda_n(U). \]
Here $F_m\not\equiv0$ because $f_m\not\equiv0$, and for $U=\{F_m>0\}$ one has $\partial U\subseteq\{F_m=0\}$. Lemma \ref{zero_set} gives $\lambda_n(\partial U)=0$, so the theorem applies. In particular, the density defining $\gamma_+(f_m)$ is an ordinary limit.

		For $\mathbb{T}^n$ by Theorem \ref{m_connect} we shall have
\[ \gamma_+(F_m)\geq\alpha(D'_m). \]
By Lemma \ref{alpha_r_T}, $\alpha(D_m)\leq\alpha(D'_m)$. Consequently,
\begin{align*}
\gamma_+(f)&\geq\gamma_+(f_m)-\delta=\gamma_+(F_m)-\delta\\
&\geq\alpha(D'_m)-\delta\geq\alpha(D_m)-\delta
 \geq\alpha(\{\theta_v\})-\delta.
\end{align*}
The last step uses $D_m\subseteq\{\theta_v\}$ and monotonicity. Letting $\delta\downarrow0$ proves (\ref{ap_equation}).
	\end{proof}
	
By Lemma \ref{prod_inf}, for $0<c_1<d_1<\cdots<c_m<d_m$, we have the real analogue of (\ref{alpha_union})
\begin{equation}\label{prod_inf_eq}\alpha\left(\bigcup_{k=1}^m [c_k;d_k]\right)\geq\prod_{k=1}^m \frac{c_k}{c_k+d_k}.\end{equation}
	We obtain a generalization of Theorem \ref{theorem_rational} by combining Theorem \ref{ap_theorem} with (\ref{prod_inf_eq}).
	\begin{theorem}
		\label{theorem_real}
If in Theorem \ref{ap_theorem}, for $0<c_1<d_1<\cdots<c_m<d_m$, we have
\[ \{\theta_v\}\subseteq\bigcup_{k=1}^m [c_k;d_k], \]
		then
\[ \gamma_+(f)\geq\prod_{k=1}^m \frac{c_k}{c_k+d_k}. \]
	\end{theorem}

	\section*{Acknowledgments} The author would like to thank Prof. Vyacheslav Spiridonov for his valuable advice on writing this paper.

 	\textbf{Funding.} The study has been funded within the framework of the HSE University Basic Research Program.

	\section*{Declaration of generative AI use}
	During the revision of this manuscript, ChatGPT (OpenAI) was used to assist in checking mathematical arguments, identifying errors and gaps, and suggesting corrections to proofs, notation, and wording. The authors retain full responsibility for the mathematical correctness and all other content of the manuscript.

	\newpage

	\begin{appendices}
	
	\section{Analytic lemmata}

	\begin{lemma}
		\label{zero_set}
		Let $p$ be a non-zero complex Laurent polynomial in several variables. The set of points on $\mathbb{T}^m$ where $p = 0$ is of zero Lebesgue measure.
	\end{lemma}
	\begin{proof}
First multiply a Laurent polynomial by a monomial to obtain an ordinary polynomial; this does not change its zeros on $\mathbb T^m$. We use induction on the number of variables. For one variable a non-zero polynomial has only finitely many zeros, so the assertion holds.

		A polynomial in $m > 1$ variables could be written as
		\[
p(z)=\sum_{j=0}^{N}z_m^j p_j(z_1,\ldots,z_{m-1}),\qquad p_N\not\equiv0,
		\]
where the $p_j$ are polynomials. By the induction hypothesis, the set where $p_N=0$ on $\mathbb T^{m-1}$ has measure zero. Whenever $p_N(a_1,\ldots,a_{m-1})\ne0$, the one-variable polynomial
		\[
			q(z_m) = p(a_1, \ldots, a_{m - 1}, z_m)
		\]
		is non-zero and only has a finite number of roots, implying that on the one-dimensional slice of $\mathbb{T}^m$
		\[
			z_1 = a_1, \,\, \ldots \,\,, z_{m - 1} = a_{m - 1}
		\]
		the roots of $p$ form a set of Lebesgue measure zero. Thus the induction step follows from the Fubini's theorem.
	\end{proof}

	\begin{lemma}
		\label{dyadic_approx}
		For a function (\ref{lemma_a_2_poly})
		\begin{align*}
			\lim_{n \to \infty} \frac{|F_+ \cap B_n|}{2^{nm}} &= \lambda_m(F_+), \\
			\lim_{n \to \infty} \frac{|F_- \cap B_n|}{2^{nm}} &= \lambda_m(F_-), \\
			\lim_{n \to \infty} \frac{|F_0 \cap B_n|}{2^{nm}} &= 0.
		\end{align*}
	\end{lemma}
	\begin{proof}
		Consider the limit for $F_+$ first. Note that by the continuity of $p$ the set $F_+$ is open and its closure $\overline{F_+}$ is contained inside $F_+ \cup F_0$.

By $I_n$ denote the union of all closed $n$-th boxes contained in $F_+$. Each $I_n$ is a subset of $F_+$. For every $q\in F_+$, openness gives a neighborhood of $q$ contained in $F_+$. Since the diameters of the boxes tend to zero, every box containing $q$ is contained in $F_+$ for all sufficiently large $n$. Therefore
\[ \mathbf 1_{I_n}(q)\longrightarrow\mathbf 1_{F_+}(q)\quad(q\in\mathbb T^m). \]
The boxes at a fixed level have pairwise disjoint interiors and measure $2^{-nm}$. No nesting between successive levels is needed. By dominated convergence,
		\begin{equation}
			\label{app_1}
			\lim_{n \to \infty} \lambda_m(I_n) = \lambda_m(F_+).
		\end{equation}
		If an $n$-th closed box lies inside $F_+$, then so does its center from $B_n$, thus we shall have
		\begin{equation}
			\label{app_2}
			\lambda_m(I_n) \leq \frac{|F_+ \cap B_n|}{2^{nm}}.
		\end{equation}

		Since $F_0$ is of measure zero,
		\[
			\lambda_m(\overline{F_+}) = \lambda_m(F_+).
		\]
		By $O_n$ denote the union of all closed $n$-th boxes that intersect $\overline{F_+}$. Each $O_n$ contains $\overline{F_+}$ inside, implying $\lambda_m(O_n) \geq \lambda_m(F_+)$. No monotonicity of $O_n$ is assumed.

		On the unit circle $\mathbb{T}$ we can define a metric $d(x, y)$ as the length of the shortest circular arc between $x$ and $y$. Similarly, on $\mathbb{T}^m$ we can define a metric
		\[
			d_m(x, y) = d(x_1, y_1) + \ldots + d(x_m, y_m).
		\]
		The largest distance between two points inside of a closed $n$-th box tends to $0$ as $n$ tends to infinity.

		By outer regularity of the Lebesgue measure $\lambda_m$ for every $\varepsilon > 0$ there is an open set $V$ of measure at most $\lambda_m(\overline{F_+}) + \varepsilon$ containing $\overline{F_+}$. Closed subsets of the compact space $\mathbb{T}^m$ are compact, thus $\overline{F_+}$ and $\mathbb{T}^m \setminus V$ are two disjoint closed compact sets in $\mathbb{T}^m$. Hence they must be at some non-zero distance from each other, i.e., there is a number $c > 0$ such that
		\[
			\forall x \in \overline{F_+}, y \in \mathbb{T}^m \setminus V\colon d_m(x, y) > c.
		\]
		And since in the sense of the metric $d_m$ the diameter of closed $n$-th boxes tends to $0$ as $n$ tends to infinity, we should have $O_n \subset V$ for all $n$ greater than some $N$, implying $\lambda_m(O_n) \leq \lambda_m(F_+) + \varepsilon$. In other words, for a large enough $n$ a closed $n$-th box could not intersect both $\overline{F_+}$ and $\mathbb{T}^m \setminus V$.

		So we shall have
		\begin{equation}
			\label{app_3}
			\lim_{n \to \infty} \lambda_m(O_n) = \lambda_m(F_+).
		\end{equation}
		Every point of $B_n$ that lies inside $F_+$ is a center of a closed $n$-th box that intersects $\overline{F_+}$, which means
		\begin{equation}
			\label{app_4}
			\frac{|F_+ \cap B_n|}{2^{nm}} \leq \lambda_m(O_n).
		\end{equation}

		Together (\ref{app_1}), (\ref{app_2}), (\ref{app_3}) and (\ref{app_4}) imply
		\[
			\lim_{n \to \infty} \frac{|F_+ \cap B_n|}{2^{nm}} = \lambda_m(F_+).
		\]
		Analogously, we can show that
		\[
			\lim_{n \to \infty} \frac{|F_- \cap B_n|}{2^{nm}} = \lambda_m(F_-).
		\]
		Since
		\[
			\frac{|F_+ \cap B_n|}{2^{nm}} + \frac{|F_- \cap B_n|}{2^{nm}} + \frac{|F_0 \cap B_n|}{2^{nm}} = 1
		\]
		and
		\[
			\lambda_m(F_+) + \lambda_m(F_-) = 1,
		\]
		we should also have
		\[
			\lim_{n \to \infty} \frac{|F_0 \cap B_n|}{2^{nm}} = 0.
		\]
	\end{proof}

	\section{Combinatorial lemmata}

	\begin{lemma}
		\label{alpha_limit}
		For a finite subset $D \subset \mathbb{N}$
		\[
			\lim_{n \to \infty} \frac{\alpha(G(n, D))}{n} = \sup_{n} \frac{\alpha(G(n, D))}{n}.
		\]
	\end{lemma}
	\begin{proof}
		Define
		\[
			\alpha_n = \frac{\alpha(G(n, D))}{n}.
		\]

		If $S = \{s_1, \ldots, s_t\}$ is a maximum independent set of $G(n, D)$, then an independent set $S^{\prime}$ of $G(kn,D)$ of size $kt$ can be obtained by repetition
		\[
			S^{\prime} = \{s_1, \ldots, s_t, s_1 + n, \ldots, s_t + n, \ldots, s_1 + (k - 1)n, \ldots, s_t + (k - 1)n \}.
		\]
		Thus
		\begin{equation}
			\label{multiple}	
			\alpha_{kn} \geq \alpha_{n}.
		\end{equation}

		Since the distance between two adjacent vertices is at most $\max(D)$, for any $m \geq \max(D)$
		\[
			S^{\prime} = \{s_1 + m, \ldots, s_t + m \}
		\]
		is an independent set of $G(n+m,D)$ of size $t$, which implies
		\begin{equation}
			\label{greater}
			\alpha_{n + m} \geq	\frac{n}{n + m} \alpha_n.
		\end{equation}

		For $m > n + \max(D)$ choose $0 \leq i < n$ such that $m - \max(D) - i$ is divisible by $n$, by inequalities (\ref{multiple}) and (\ref{greater})
		\begin{equation}
			\label{m_n}
			\alpha_m \geq \frac{m - \max(D) - i}{m} \alpha_{m - \max(D) - i} \geq \left(1 - \frac{\max(D) + n}{m} \right)\alpha_n.
		\end{equation}

		Take supremum
		\[
			\alpha^{\prime} = \sup_{n} \frac{\alpha(G(n, D))}{n}.
		\]
		For every $\varepsilon \in (0; 1)$ there is a number $n$ such that $\alpha_n \geq (1 - \varepsilon)\alpha^{\prime}$, and for all sufficiently large $m$ inequality (\ref{m_n}) gives us
		\[
			\alpha_m \geq (1 - \varepsilon) \alpha_n \geq (1 - \varepsilon)^2 \alpha^{\prime},
		\]
		\[
			(1 - \varepsilon)^2 \alpha^{\prime} \leq \alpha_m \leq \alpha^{\prime}.
		\]
		Since the choice of $\varepsilon$ here was arbitrary, we arrive at
		\[
			\lim_{m \to \infty} \alpha_m = \alpha^{\prime}.
		\]
	\end{proof}

	\begin{lemma}
		\label{ap_alpha_limit}
		The following limit
		$$ \lim_{l \to \infty} \frac{\alpha(l, D)}{l} $$
		exists for bounded $D \subset \mathbb{R}_+$ and is equal to $\limsup_{l \to \infty} \frac{\alpha(l, D)}{l}$.
	\end{lemma}
	\begin{proof}
If $D=\varnothing$, then $\alpha(l,D)=l$ and the assertion is immediate. Otherwise denote $\limsup_{l\to\infty}\alpha(l,D)/l$ by $\alpha$. If $\alpha=0$, non-negativity already proves convergence. Assume $\alpha>0$ and fix $0<\varepsilon<\alpha$. There are arbitrarily large $L$ such that
		\[ \alpha(L, D) \geq \left(\alpha - \frac{\varepsilon}{2}\right)L. \]
		Since $\alpha(L, D)$ is a supremum over independent subsets of $(0; L)$, there is an open $D$-set $A \subset (0; L)$ of measure at least $(\alpha - \varepsilon)L$. We put $m$ copies of $A$ next to each other with empty segments of length $\sup(D)$ inserted in-between to guarantee that the resulting set would still be independent: two elements $x$ and $y$ belonging to different copies would be at distance strictly greater than $\sup(D)$ and thus would satisfy the condition $x - y \notin D$. In other words, we take
		\[ B_m = A \cup (A + (L + \sup(D))) \cup \ldots \cup (A + (m-1)(L + \sup(D))). \]
		It is an independent subset of $(0; (L + \sup(D))m)$ of measure $m \mu(A)$, so
		\[ \alpha((L + \sup(D))m, D) \geq m(\alpha - \varepsilon) L. \]
Write $l=m(L+\sup D)+l_0$, where $m=\lfloor l/(L+\sup D)\rfloor$ and $0\leq l_0<L+\sup D$. Considering $B_m\subset(0;l)$ gives
		\[ \alpha(l, D) \geq m (\alpha - \varepsilon) L, \]
\[ \frac{\alpha(l,D)}{l}\geq(\alpha-\varepsilon)\frac{mL}{m(L+\sup D)+l_0}. \]
For fixed $L$, let $l\to\infty$ to obtain
\[ \liminf_{l\to\infty}\frac{\alpha(l,D)}{l}\geq(\alpha-\varepsilon)\frac{L}{L+\sup D}. \]
Let $L\to\infty$ along the arbitrarily large choices above and then $\varepsilon\downarrow0$. This yields $\liminf\alpha(l,D)/l\geq\alpha=\limsup\alpha(l,D)/l$, proving existence of the limit.
	\end{proof}

	\begin{proposition}
		\label{union}
		For two finite subsets $D_1, D_2 \subset \mathbb{N}$
		\[
			\alpha(D_1 \cup D_2) \geq \alpha(D_1) \alpha(D_2).
		\]
	\end{proposition}
	\begin{proof}
		By $G_n, G_n^{(1)}$ and $G_n^{(2)}$ denote the circulant graphs on $n$ vertices with sets of distances $D_1 \cup D_2$, $D_1$, and $D_2$, respectively. Let $A_1$ and $A_2$ be maximum independent subsets of $G_n^{(1)}$ and $G_n^{(2)}$, respectively.

		Define a right circular shift by $1$ position as
		\[
			S(i) = \begin{cases}
				i + 1, & \textrm{if } i < n - 1, \\
				0, & \textrm{if } i = n - 1.
			\end{cases}
		\]
		For a set of vertices $U = \{u_1, \ldots, u_k\} \subseteq V$ by $S(U)$ we mean the set $\{S(u_1), \ldots, S(u_k)\}$. A right circular shift can be applied several times
		\[
			S^k(U) = S(\ldots S(U) \ldots).
		\]

		Each vertex $v$ appears exactly $|A_2|$ times among the sets \[
			A_2, S(A_2), \ldots, S^{n - 1}(A_2),
		\]
		thus
		\[
			\sum_{j} |A_1 \cap S^{j}(A_2)| = |A_1| \cdot |A_2|.
		\]
		Therefore there is a number $j$ such that
		\[
			|A_1 \cap S^{j}(A_2)| \geq \frac{|A_1| \cdot |A_2|}{n}.
		\]

		No pair of vertices in the set $S^j(A_2)$ could be at some distance $d \in D_2$, and no pair of vertices in the set $A_1$ could be at some distance $d \in D_1$. So $A_1 \cap S^j(A_2)$ is an independent set of $G_n$, and
		\[
			\frac{\alpha(G_n)}{n} \geq \frac{|A_1 \cap S^j(A_2)|}{n} \geq \frac{|A_1| \cdot |A_2|}{n^2} = \frac{\alpha(G_n^{(1)})}{n}\frac{\alpha(G_n^{(2)})}{n}.
		\]
		By taking limit we arrive at conclusion
		\[
			\alpha(D_1 \cup D_2) \geq \alpha(D_1) \alpha(D_2).
		\]
	\end{proof}

	\begin{lemma}
		\label{m_alpha_limit}
		For a finite set $D$ of the form (\ref{set_D})
		\[
			\lim_{n \to \infty} \frac{\alpha(G(n, D))}{n^m} = \sup_{n} \frac{\alpha(G(n, D))}{n^m}.
		\]
	\end{lemma}
	\begin{proof}
		Let
		\[
			\alpha_n = \frac{\alpha(G(n, D))}{n^m}.
		\]
		Let $S \subseteq \{0, \ldots, n - 1\}^m$ be a maximum independent set of $G(n,D)$. Note that
		\[
			S + \{0, n, \ldots, (t - 1)n \}^m
		\]
		is an independent set of $G(tn,D)$ of size $t^m |S|$, thus
		\[
			\alpha_{tn} \geq \alpha_{n}.
		\]

Put $R=\max_{d\in D} \max_{i} |d_i|$. The same set $S$ is independent in $G(n+k,D)$ for every integer $k\geq R$: a new modular edge would require a coordinate difference of magnitude at least $(n+k)-R\geq n$, whereas all coordinate differences in $S$ have magnitude at most $n-1$. Thus
		\[
			\alpha_{n + k} \geq \left(\frac{n}{n + k}\right)^m \alpha_n.
		\]

		The rest of the proof is analogous to the proof of Lemma \ref{alpha_limit}.
	\end{proof}

	\begin{lemma}
		\label{alpha_r_T}
		For the set of differences (\ref{fin_set_D}) we have an inequality\footnote{Only this inequality is needed below.}
		\begin{equation} \label{alpha_r_T_equation} \alpha(D) \leq \alpha(\{ \pm \vec{d}_j \}), \end{equation}
		vectors $\vec{d}_j$ are taken from (\ref{vec_d_j}).
	\end{lemma}
	\begin{proof}
If $\alpha(D)=0$, the claim is immediate. Otherwise fix $0<\varepsilon<\alpha(D)$ and, for sufficiently large $L$, take an open $D$-set $A\subset(0;L)$ such that
	\[ \mu(A) \geq (\alpha(D) - \varepsilon) L. \]
	Take a number $M \in (\sup(D); 2\sup(D))$ such that $L + M = T$ together with $a_1, \ldots, a_n$ forms a $\mathbb{Z}$-linearly independent set of numbers. Such a choice is possible because only countably many values of $T$ are excluded. We now consider a set
	\[ B = A + \{\ldots, -2T, -T, 0, T, 2T, \ldots \}, \]
which is still a $D$-set: differences between points in distinct translates have magnitude greater than $T-L=M>\sup D$.

	To each vector $\vec{x} \in \mathbb{Z}^n$ we can put in correspondence
	\[ p(\vec{x}) = x_1 a_1 + \ldots + x_n a_n. \]
	Take a subset $I = p^{-1}(B)$. Note that $I$ would be an independent set in the infinite lattice with forbidden differences $D^{\prime}=\{\pm\vec d_j\}$. Indeed, assume we have a pair of vectors $\vec{x}, \vec{y} \in I$ such that $\vec{x} - \vec{y} = \vec{d}_j$. Since our map $p$ is linear, we shall have
	\[ p(\vec{x}) - p(\vec{y}) = p(\vec{d}_j) = d_{j1} a_1 + \ldots + d_{jn} a_n = \alpha_j. \]
	Together with $p(\vec{x}), p(\vec{y}) \in B$ this contradicts independence of $B \subset \mathbb{R}$.

	With each $\vec{x} \in \mathbb{Z}^n$ we associate a point in $[0; T)$
	\[ q(\vec{x}) \equiv p(\vec{x}) \, \textrm{\,mod\,}  T, \]
	i.e., $q(\vec{x}) = p(\vec{x}) - zT$ with $z \in \mathbb{Z}$ and $q(\vec{x}) \in [0; T)$.

	Consider a box $X_k = \{0, 1, \ldots, k - 1\}^n$, we are going to show
\begin{equation}\label{lattice_density}\liminf_{k\to\infty}\frac{|I\cap X_k|}{k^n}\geq\frac{\mu(A)}{T}.\end{equation}
	
	We want to estimate the number of points in $X_k$ mapped into $A$ by $q$, since $A$ is an open set and thus a union of pairwise disjoint open intervals
	\begin{equation} \label{A_l_r} A = (l_1; r_1) \cup (l_2; r_2) \cup \ldots, \end{equation}
	we first show
	\begin{lemma}
		For a non-empty interval $(l;r)\subset(0;L)$ we have
		$$ \lim_{k \to \infty} \frac{|q^{-1}((l; r)) \cap X_k|}{k^n} = \frac{r - l}{T}. $$
	\end{lemma}
	\begin{proof}
For $h\in\mathbb Z\setminus\{0\}$ the normalized exponential sum factors as
\[
\frac1{k^n}\sum_{\vec x\in X_k}e^{2\pi i h q(\vec x)/T}
 =\prod_{j=1}^n\left(\frac1k\sum_{r=0}^{k-1}e^{2\pi i h a_jr/T}\right).
\]
Because $a_1/T$ is irrational, the first factor is a geometric sum whose modulus is at most
\[ \frac{2}{k|1-e^{2\pi i h a_1/T}|}\longrightarrow0. \]
All other factors have modulus at most one. Hence these averages tend to zero for every non-zero integer $h$. They therefore converge to normalized Lebesgue integrals for trigonometric polynomials on $\mathbb R/T\mathbb Z$, and, by uniform approximation, for continuous functions. Bounding the indicator of $(l;r)$ above and below by continuous functions whose integrals differ by arbitrarily little gives
\begin{equation}\label{varphi_k}
\lim_{k\to\infty}\frac{|q^{-1}((l;r))\cap X_k|}{k^n}=\frac{r-l}{T}.
\end{equation}
This proves the lemma.
	\end{proof}
	
	If $A$ is a finite union of intervals the inequality (\ref{lattice_density}) is immediate. Otherwise, by $\sigma$-additivity of the Lebesgue measure for any $\varepsilon_2 > 0$ we find a number $m$ such that
	$$ \mu((l_1; r_1) \cup ... \cup (l_m; r_m)) \geq \mu(A) - \varepsilon_2. $$
	Applying our lemma we obtain inequality
\begin{align*}
\liminf_{k\to\infty}\frac{|I\cap X_k|}{k^n}
&\geq\lim_{k\to\infty}\sum_{j=1}^m\frac{|q^{-1}((l_j;r_j))\cap X_k|}{k^n}\\
&=\frac{\mu(\bigcup_{j=1}^m(l_j;r_j))}{T}
\geq\frac{\mu(A)-\varepsilon_2}{T}.
\end{align*}
	Since $\varepsilon_2$ is arbitrary, we conclude that inequality (\ref{lattice_density}) holds.

Put $D'=\{\pm\vec d_j\}$ and $R=\max_{d\in D'}\|d\|_\infty$. The set $I\cap X_k$ need not be independent modulo $k$. It is, however, independent in $G(k+R,D')$: an extra modular edge would require a coordinate difference of magnitude at least $(k+R)-R=k$, which is impossible inside $X_k$. Therefore
\[ \alpha(D')\geq\frac{|I\cap X_k|}{(k+R)^n}. \]
Using (\ref{lattice_density}) and letting $k\to\infty$ gives $\alpha(D')\geq\mu(A)/T$. Consequently,
	$$ \alpha(\{ \pm \vec{d}_j \}) \geq (\alpha(D) - \varepsilon) \frac{L}{L+M} \geq (\alpha(D) - \varepsilon) \frac{L}{L+2\sup(D)}. $$
	But in the above inequality $\varepsilon$ can be chosen arbitrarily small and $L$ arbitrarily large, from which (\ref{alpha_r_T_equation}) follows.
	
\end{proof}

	\begin{lemma}
		\label{an_consec}
		For $D = [a; b]$ with $b > a > 0$ one has
		\[ \alpha(D) \geq \frac{1}{1 + \frac{b}{a}} = \frac{a}{a + b}. \]
	\end{lemma}
	\begin{proof}
		For an independent set
		\[ (0; a) \cup (a + b; (a+b) + a) \cup \ldots \cup (k(a+b); k(a+b) + a) \]
		we have
		\[
			\frac{\alpha(k(a+b) + a, D)}{k(a+b) + a} \geq \frac{(k + 1) a}{k(a+b) + a},
		\]
		which in limit gives us
		\[
			\limsup_{l \to \infty} \frac{\alpha(l, D)}{l} \geq \frac{a}{a + b}.
		\]
	\end{proof}

	\begin{lemma}
		\label{an_union}
		For two bounded subsets $D_1, D_2 \subset \mathbb{R}_+$
		\begin{equation}\label{an_union_eq} \alpha(D_1 \cup D_2) \geq \alpha(D_1) \alpha(D_2). \end{equation}
	\end{lemma}
	\begin{proof}
If either $D_i$ is empty, the claim is immediate. If either $\alpha(D_i)=0$, the right hand side is zero and the claim also follows. Assume both densities are positive and fix $0<\varepsilon<\min_i\alpha(D_i)$. Put $R=\max(\sup D_1,\sup D_2)$. By Lemma \ref{ap_alpha_limit}, for arbitrarily large $L'$ there are open $D_i$-sets in $(0;L')$ with measure arbitrarily close to $\alpha(D_i)L'$. Taking finitely many of their disjoint interval components gives finite unions of open intervals $A_i\subset(0;L')$ satisfying
\[ \mu(A_i)\geq(\alpha(D_i)-\varepsilon/2)L',\qquad i=1,2. \]
Set $L=L'+R$. For sufficiently large $L'$ these sets satisfy
\[ \mu(A_i)\geq(\alpha(D_i)-\varepsilon)L,
\qquad A_i\cap[L-R;L)=\varnothing,\qquad i=1,2. \]
Thus the required padding holds for both sets, not just for $A_1$.

		Because of the conditions discussed in the above paragraph, the sets
		\[ B_i = A_i \cup (A_i + L) \cup \ldots \cup (A_i + L(k - 1)) \subset (0; kL) \]
		are also independent. Now for a number $x \in [0; L]$ the set
		\[ Y_x = (B_1 + x) \cap B_2 \subset (0; kL)\]
		would be independent with respect to $D_1 \cup D_2$. This holds true because $B_1 + x$ is independent with respect to $D_1$ as a shift, subsets of independent sets are independent, and our set $Y_x$ is a subset of both $B_1 + x$ and $B_2$. 

		Because $B_i$ are finite unions of intervals, the function $\mu(Y_x)$ is continuous. One would want to estimate
		\[ \max_{x \in [0; L]} \mu(Y_x). \]
		By an indicator function $\chi_A(x)$ we mean a function that is equal to $1$ on $A$ and $0$ outside. We could note that
\begin{align*}
\int_0^L\mu(Y_x)\,dx
&=\int_0^L\int_{\mathbb R}\chi_{B_1}(y-x)\chi_{B_2}(y)\,dy\,dx\\
&=\int_{\mathbb R}\chi_{B_2}(y)\int_0^L\chi_{B_1}(y-x)\,dx\,dy\\
&\geq\int_L^{kL}\chi_{B_2}(y)\mu(B_1\cap(y-L;y))\,dy.
\end{align*}
		Since $B_1$ as a set is $L$-periodic inside $(0; kL)$ and $(y - L; y) \subset (0; kL)$, we can replace term $\mu(B_1 \cap (y - L; y))$ with  $\mu(A_1)$
		\[
			\int_{0}^{L} \mu(Y_x) dx \geq \mu(A_1) \int_{L}^{kL} \chi_{B_2}(y) dy = (k - 1) \mu(A_1) \mu(A_2).
		\]
		We get
		\[
			\max_{x \in [0;L]} \mu(Y_x) \geq \frac{1}{L}\int_{0}^{L} \mu(Y_x) dx \geq \frac{k - 1}{L} \mu(A_1) \mu(A_2).
		\]
		Since $Y_x$ is an independent set for $D_1 \cup D_2$, we should have
\begin{align*}
\frac{\alpha(kL,D_1\cup D_2)}{kL}
&\geq\frac{k-1}{k}\frac{\mu(A_1)}{L}\frac{\mu(A_2)}{L}\\
&\geq\frac{k-1}{k}(\alpha(D_1)-\varepsilon)(\alpha(D_2)-\varepsilon).
\end{align*}
		Since $\varepsilon$ here was chosen arbitrary and $k$ can be taken to be arbitrarily large, we arrive at (\ref{an_union_eq}).
	\end{proof}

	Combining Lemmas \ref{an_consec} and \ref{an_union} we get
	\begin{lemma}
		\label{prod_inf}
		For a sequence $0 < a_1 < b_1 < \ldots < a_m < b_m$ we shall have
		\[
			\alpha([a_1; b_1] \cup [a_2; b_2] \cup \ldots) \geq \prod_{i=1}^{m} \frac{1}{1 + \frac{b_i}{a_i}}.
		\]
	\end{lemma}

	\end{appendices}

\end{document}